\documentclass[12pt, twoside]{article}
\usepackage{amsmath,amsthm,amssymb}
\usepackage{times}
\usepackage{enumerate}
\usepackage{bm}
\usepackage[all]{xy}
\usepackage{mathrsfs}
\usepackage{amscd}
\usepackage{color}
\def\titlerunning#1{\gdef\titrun{#1}}
\makeatletter
\def\author#1{\gdef\autrun{\def\and{\unskip, }#1}\gdef\@author{#1}}
\def\address#1{{\def\and{\\\hspace*{18pt}}\renewcommand{\thefootnote}{}%
\footnote {#1}}%
\markboth{\autrun}{\titrun}}
\makeatother
\def\email#1{e-mail: #1}
\def\subjclass#1{{\renewcommand{\thefootnote}{}%
\footnote{\emph{Mathematics Subject Classification (2020):} #1}}}
\def\keywords#1{\par\medskip
\noindent\textbf{Keywords.} #1}

\newtheorem{theorem}{Theorem}[section]

\newtheorem{lemma}[theorem]{Lemma}
\newtheorem{proposition}[theorem]{Proposition}
\theoremstyle{definition}

\newtheorem{remark}[theorem]{Remark}

\numberwithin{equation}{section}

\newtheorem{question}[theorem]{Question}

\xyoption{all}

\begin{document}
\baselineskip=17pt

\titlerunning{The interpolation method for the complex Hessian equation on compact K\"ahler manifolds}
\title{The interpolation method for the complex Hessian equation on compact K\"ahler manifolds}

\author{Jiaogen Zhang and Xi Zhang}

\date{}

\maketitle

\address{J. Zhang: School of Mathematics, Hefei University of Technology, Hefei 230009, P.R. China;
\email{zjgmath@huft.edu.cn}}

\address{X. Zhang: School of Mathematics and Statistics, Nanjing University of Science and Technology, Nanjing 210094, China;
\email{mathzx@njust.edu.cn}
}

\subjclass{35B45; 46B70; 32W50; 32Q15}

\begin{abstract}
In current work, we  delve into the study of complex Hessian equations on compact K\"ahler manifolds. By imposing a constraint on both the bisectional curvature and the oscillation of the solution, we demonstrate that the Laplacian estimate can be linearly controlled by its gradient estimate. Consequently, we are able to derive gradient estimate for the solution using a standard interpolation argument, without the need to employ the blow-up technique.  As an application, we demonstrate that when the right-hand side of the complex Hessian equations approximate 1 in the $L^1$ norm, the desired gradient estimate is obtained. 
\end{abstract}
\keywords{Complex Hessian equation, gradient estimate, interpolation theory, K\"ahler manifold}

\section{Introduction}
Let $(M,\omega)$ be a compact K\"ahler manifold with  $\mathrm{dim}_{\mathbb{C}}M=n$. 
For every fixed integer $k=1,\cdots, n$, we are interested in the following complex $k$-th Hessian equation
\begin{equation}\label{complex Hessian equation 1}
\begin{cases}
\omega_{u}^{k}\wedge \omega^{n-k}=h\omega^{n}\,, \\[1mm]
\omega_{u}=\omega+\sqrt{-1} \partial \overline{\partial} u\in\Gamma_{k}(M,\omega)\,.
\end{cases}
\end{equation}
Here, the $ k $-positive cone $\Gamma_{k}(M,\omega)$ is the set of $(1,1)$ forms whose eigenvalues, with respect to $\omega$, lie in the convex cone
\[
\Gamma_k=\{ \lambda \in \mathbb{R}^n \mid \sigma_1(\lambda)>0\,,\dots,\sigma_k(\lambda)>0 \}
\]
with the $j$-th elementary symmetric polynomial defined by
\[
\sigma_{j}(\lambda)=\sum_{1\leq i_{1}<i_{2} <\cdots i_{j}\leq n}\lambda_{i_1}\lambda_{i_2}\cdots \lambda_{i_{j}}, \quad \lambda=(\lambda_1,\lambda_{2}, \cdots, \lambda_{n}),
\]
and $h$ is a positive smooth function on $M$ satisfying the compatibility condition:
\begin{equation}
\int_{M}(h-1)\omega^{n}=0\,.
\end{equation}
A function $u\in C^{2}(M,\mathbb{R})$ is called $k$-admissible if $\omega_{u}\in \Gamma_{k}(M,\omega)$. When $k=1$ this reduces to the Laplacian equation. When $k=n$, equation \eqref{complex Hessian equation 1} reduces to the complex Monge-Amp\`ere equation, which has been thoroughly investigated following the groundbreaking work of Yau \cite{yau}. Therefore, in this work, we primarily concentrate on the intermediate cases where $2\leq  k\leq  n-1$.\medskip

Fully nonlinear partial differential equations (PDEs) that incorporate elementary symmetric polynomials arise frequently in differential geometry. Among these, complex Hessian equations are an important class of PDEs with applications in complex geometry, mathematical physics, and complex dynamical systems. Many researches have focused on understanding the existence, uniqueness, and properties of solutions to these equations, as well as developing new techniques for solving them. A notable example is the Fu-Yau equation \cite{FY07,FY08} arising from the Strominger system in complex geometry. In higher dimensions, the Fu-Yau equation corresponds precisely to a complex second Hessian type equation and has been extensively studied by Chu-Huang-Zhu \cite{CHZ19a,CHZ19b} and by Phong-Picard-Zhang \cite{PPZ16,PPZ17,PPZ19,PPZ21}. Another representative example  is the quaternionic Calabi conjecture on hyperK\"ahler with torsion (HKT) manifolds, proposed by Alesker and Verbitsky \cite{AV10}. In their work, the quaternionic Monge-Amp\`ere equation is reformulated as a special case of the complex Hessian equation on complex manifolds.
\medskip

In the real setting, the Hessian equation has been extensively studied in a considerable amount of literature (see, for instance, \cite{CNS,Chou-Wang,Guan,Li,Tru95,Urbas,Wang} and the references therein). Compared to the real case, the complex setting presents additional difficulties, particularly in controlling higher-order terms arising from complex conjugation. Following closely the case-wise arguments of Chou-Wang \cite{Chou-Wang}, who dealt with the real Hessian type equations, Hou-Ma-Wu \cite{Hou-Ma-Wu} provided a second order estimate on a closed K\"ahler manifold, specifically in the following form:
\begin{equation}\label{HMW}
\|\Delta_{\omega}  u\|_{L^{\infty}}\leq  C\big(1+\|\nabla u\|_{L^{\infty}}\big)^{2}
\end{equation}
for a uniform positive constant $C$. Building on this particular estimate, Dinew and Ko\l{}odziej \cite{DK} were able to obtain the gradient estimate of \eqref{complex Hessian equation 1}.  Eventually, Sz\'{e}kelyhidi \cite{Sze} and Zhang \cite{Zhang} independently proved the solvability of \eqref{complex Hessian equation 1} on compact Hermitian manifolds, while Chu-Huang-Zhang \cite{CHZ,HZ} demonstrated the same for compact almost Hermitian manifolds. Recently, Cheng and Xu \cite{CX25} obtained an upper approximations for $m$-subharmonic functions on compact K\"ahler manifolds with nonnegative holomorphic bisectional curvature. This result guarantees the H\"older regularity of solutions to equation \eqref{complex Hessian equation 1} when $h \in L^{p}(M)$ for any $p > \frac{n}{k}$. \medskip

The Dirichlet problem in a bounded smooth domain of $\mathbb{C}^n$  was solved by Vinacua \cite{Vin} and Li \cite{LSY04}. In a recent work \cite{CP}, Collins and Picard considered the Dirichlet problem on the compact Hermitian manifold  with smooth boundary. The pluripotential theory for complex Hessian operators was given by B{\l}ocki \cite{Blo05a}. Since then, the weak solution for complex Hessian equation has been well developed, see, e.g. \cite{CX25,DK14,KN,LV} and reference therein. \medskip

The gradient estimate for complex Hessian equation  presented in \cite{DK} relies on a blow-up argument. However, acquiring an apriori gradient estimate for complex Hessian equations on complex manifolds through the application of the maximum principle remains a recognized unsolved challenge in this field. To the best of our knowledge, only a limited number of results are currently known. For instance, assuming non-negativity of the holomorphic bisectional curvature, Hou \cite{Hou}, Jbilou \cite{Jbilou}, and Kokarev \cite{Kokarev} established the a priori gradient estimate as they attempted to solve \eqref{complex Hessian equation 1}. Additionally, Zhang \cite{XZhang} derived apriori gradient estimate for \eqref{complex Hessian equation 1} under the supplementary assumption that $\omega_u $ belongs to $ \Gamma_{n}(M,\omega).$ \medskip

In contrast to the blow-up argument presented in \cite{DK}, we attempt to establish the gradient estimate for the complex Hessian equation \eqref{complex Hessian equation 1} through the application of a standard interpolation technique. However, this approach necessitates a more stringent second order estimate than \eqref{HMW} as in \cite{Hou-Ma-Wu}, as outlined below:
\begin{equation}\label{lambda and nabla}
\|\Delta_{\omega}  u\|_{L^{\infty}}\leq  C\big(1+\|\nabla u\|_{L^{\infty}}\big).
\end{equation}
Therefore, a natural question emerges as follows:
\begin{question}\label{linear control}
Let $ u$ be a smooth $k$-admissible solution of \eqref{complex Hessian equation 1} for $k=2,\cdots,n-1$. Does there exists a uniform constant $C$ for which \eqref{lambda and nabla} holds\,?
\end{question}

In current work, we attempt to provide a partial answer to Question \ref{linear control}. Before presenting our main result, it is necessary to revisit a specific constant, as described in \cite{Hou-Ma-Wu}, which is determined in terms of the bisectional curvature of the underlying manifold $(M,\omega)$. More precisely, for every point $x\in M$, let
\[
R_{\mu \overline{\mu} \nu \overline{\nu}}(x)=R(\mu, \overline{\mu}, \nu, \overline{\nu})(x), \qquad \mu\,,\,\nu\in T^{1,0}_{x}M
\]
be the holomorphic bisectional curvature of $(M,\omega)$ at $x$. Then, we introduce a nonnegative curvature constant $\mathcal{R}_{\omega}$ by
\begin{equation*}
\mathcal{R}_{\omega}=
\left\{\begin{array}{ll}
0, \quad &\textrm{if the holomorphic bisectional curvature is non-negative,} \\[1mm]
\underset{x\in M}{\sup}\, \big|\underset{\mu,\nu}{\inf}\, R_{\mu \overline{\mu} \nu \overline{\nu}}(x) \big|,  \quad &\textrm{otherwise}
\end{array}\right.
\end{equation*}
in which the infimum runs through two unit vectors $\mu,\nu$ in $T^{1,0}_{x}M$.

Our main result is the following theorem:
\begin{theorem}\label{main thm}
Let $ u$ be a smooth $k$-admissible solution of \eqref{complex Hessian equation 1}  for $k=2,\cdots,n-1$. Suppose that the bisectional curvature of $(M,\omega)$ and the oscillation of  $u$ on $M$ satisfy the following inequality
\begin{equation}\label{curvature}
\mathcal{R}_{\omega}\mathrm{osc}_{M}u<\frac{1}{2(2+\log 2)} \approx  0.2173. 
\end{equation} 
Then, there exists a constant $C$ that depends on $n$, $k$, $h$, $\mathrm{osc}_{M}u$, and $(M,\omega)$  such that the inequality \eqref{lambda and nabla} is satisfied. 
\end{theorem}
\begin{remark}
We notice that the upper bound $\frac{1}{2(2+\log 2)}$ is not sharp in general, whereas we will not pursue it in this work.  It is crucial to observe that the quantity $\mathcal{R}_{\omega}\mathrm{osc}_{M}u$ remains unchanged under scaling transformations of the metric $\omega$. Hence, removing the dependence of $\mathrm{osc}_{M}u$  in inequality \eqref{curvature} offers a promising prospect. Suppose this can be achieved, particularly if 
\begin{equation}\label{curvature new}
\mathcal{R}_{\omega}<\frac{1}{2(2+\log 2)}
\end{equation}
and we consequently derive \eqref{lambda and nabla}, then by rescaling the metric $\omega$, the newly formulated condition in \eqref{curvature new} will be satisfied, thereby comprehensively resolving Question \ref{linear control}.
\end{remark}

In particular, as $h$ approaches $1$ in the $L^{p}$ norm with $p\geq  1$,  we can leverage the stability of solutions to the complex Hessian equations, as demonstrated by Dinew and Kolodziej in \cite{DK14} for $p>1$ and by Guo, Phong, and Tong in \cite{GPT23} for $p=1$, the oscillation of  $u$ on $M$  tends to 0. Consequently, this automatically satisfies the hypothesis \eqref{curvature}. Hence, we can conclude the following:
\begin{theorem}\label{h closes 1}
Let $ u$ be a smooth $k$-admissible solution of \eqref{complex Hessian equation 1}  for $k=2,\cdots,n-1$. Assume that the following inequality holds for a certain uniform constant $\epsilon_{0}$:
\begin{equation}
\|h-1\|_{L^{1}(M)}< \epsilon_{0}\,.   
\end{equation}
Then, there exists a constant $C$ that depends on $n$, $k$, $h$, $\mathrm{osc}_{M}u$, and $(M,\omega)$  such that the inequality \eqref{lambda and nabla} is satisfied.
\end{theorem}

Generally speaking, the estimate \eqref{lambda and nabla} is quiet useful in the investigation of complex Hessian equations in complex geometry.  As an illustrative application, we  will employ a standard interpolation technique to deduce the following gradient estimate. 
\begin{theorem}\label{gradient estimate} Under the same assumptions as in either Theorem  \ref{main thm} or Theorem \ref{h closes 1}. Then, there exists a  constant $C$ that depends on $n$, $k$, $h$, $\mathrm{osc}_{M}u$, and $(M,\omega)$ such that 
\begin{equation}
\|\nabla u\|_{L^{\infty}}+\|\Delta_{\omega}  u\|_{L^{\infty}} \leq  C.
\end{equation}
\end{theorem}\medskip

\begin{remark}
The interior $W^{2,p}$ estimate for the Dirichlet problem of the real Monge-Amp\`ere equation 
\[
\det (u_{ij})=h
\]
was derived in the seminal work of Caffarelli \cite{Caf}, under the condition that $h$ is a small perturbation of $1$ (i.e. $\|h-1\|_{L^{\infty}} < \epsilon$ for some small constant $\epsilon$). Cheng and Xu \cite{CX23} subsequently obtained an analogous result for the complex Monge-Amp\`ere equation. Theorem \ref{gradient estimate} can be regarded as a global counterpart of these interior estimates, proving global $C^2$ estimates for complex Hessian operators when $h$ is close to $1$ in an $L^1$ sense.
\end{remark}

The rest of the article is organised as follows: In Section \ref{proof of main Thm}, we will establish the second order estimate for the equation \eqref{complex Hessian equation 1}, thereby concluding the proof of Theorem \ref{main thm}. This type of bound is appropriate for employing a standard interpolation argument, which provides an estimate for the gradient. This is detailed in Section \ref{gradient section}. The purpose of Section \ref{h closes 1 section} is to validate the assumption stated in Theorem \ref{main thm} under the hypothesis that $h$ approaches appropriately  to 1 in the $L^{1}$ norm.
\medskip

\section{Proof of Theorem \ref{main thm}}\label{proof of main Thm}

We consider the concave operator
\[
F(\omega_{u})=\big(\sigma_{k}(\omega_{u})\big)^{\frac{1}{k}}=\Big(\frac{\omega_{u}^{k}\wedge \omega^{n-k}}{\omega^{n}}\Big)^{\frac{1}{k}}, \qquad \omega_{u}\in \Gamma(M,\omega)\]
and denote its first and second partial derivatives by
\[
F^{i\overline j}=\frac{\partial F}{\partial u_{i\overline j}}\,, \qquad F^{i\overline j,k\overline{l}}=\frac{\partial^{2} F}{\partial u_{i\overline j}\partial u_{k\overline l}}\,.
\]
For every $v\in C^{2}(M,\mathbb{R})$ we define the linearized operator associated with $F$ as
\begin{equation*}
L(v)=\sum_{i,j=1}^{n} F^{i\overline{j}}v_{i\overline{j}}\,.
\end{equation*}
Note that the positive definite matrix $(F^{i\overline j})$  becomes diagonal at points where the tensor $\omega_{u}$ is diagonal. In such a case we have 
\begin{equation}\label{Fik}
F^{i\overline j}=\delta_{ij}F^{i\overline i}=\frac{1}{k}\sigma_{k}(\omega_{u})^{\frac{1}{k}-1}\frac{\partial \sigma_{k}(\lambda)}{\partial \lambda_{i}}(\omega_{u})\delta_{ij}\,.
\end{equation}

The proof of Theorem \ref{main thm} hinges crucially on the application of the maximum principle. To implement it effectively, we construct a new auxiliary function $Q$, which differs substantially from those defined in the existing literature.
\subsection{Construction of auxiliary function}
In what follows, we always assume that the solution to  \eqref{complex Hessian equation 1} is non-constant, which implies that the oscillation of $u$ is non-zero, i.e. $\mathrm{osc}_Mu>0$. This is because, in the case where $u$ is a constant, all higher-order estimates would inevitably be zero.

Let $\lambda = (\lambda_{1}, \cdots, \lambda_{n})$ denote the complex eigenvalues of $\omega_u$ with respect to $\omega$, and suppose they are arranged in decreasing order  
\begin{equation}\label{eigenvalue decreasing}
\lambda_{1}\geq \lambda_{2}\geq \cdots\geq \lambda_{n}\,. 
\end{equation}
Consider the following two useful quantities:
\begin{equation}\label{KD}
\Lambda=\sup_{M}\lambda_{1}\,, \qquad K=\sup_{M}|\nabla u|_{\omega}+D\,,
\end{equation}
where $D$ is a uniform positive constant that will be determined later. We introduce the following auxiliary function
\begin{equation}
Q =2\sqrt{\frac{\lambda_{1}}{\Lambda}}+\varphi(|\nabla u|_{\omega}^{2})+\psi(u)
\end{equation}
which defined on $M$, where
\begin{equation*}
\begin{split}
\varphi(s) &= -\frac{1}{\beta}\log(2K^{2}-s)\,, \quad \qquad \qquad s\in [0,K^2]\,,\\
\psi(t)&=-\frac{1}{\beta}\log\Big(\mathrm{osc}_{M}u+t-\inf_{M}u \Big)\,, \quad t\in \Big[\inf_{M}u,\, \sup_{M}u\Big].
\end{split}
\end{equation*}
Here and hereafter $\beta=1+\log 2$ is a specific constant.
It is clear that
\begin{equation}\label{xieta}
\mathrm{osc}_{M}\varphi(|\nabla u|_{\omega}^{2})\leq  \frac{1}{\beta}\log 2\,, \quad \varphi''=\beta(\varphi')^{2}\,,  \quad  \frac{1}{ \beta K^{2}} \geq  \varphi' \geq  \frac{1}{2\beta K^{2}}
\end{equation}
as well as
\begin{equation}\label{varphi}
\mathrm{osc}_{M}\psi(u)\leq  \frac{1}{\beta}\log 2\,, \quad \psi''=\beta(\psi')^{2}\,, \quad \frac{1}{\beta\mathrm{osc}_{M} u}\geq  -\psi'\geq  \frac{1}{2\beta\mathrm{osc}_{M} u}\,.    
\end{equation}

Let $x_{0}\in M$ be the point where the auxiliary function $Q$ attains its maximum. To prove Theorem \ref{main thm},  it is sufficient to  show
\begin{equation}\label{goal}
\lambda_{1}(x_{0}) \leq  CK.
\end{equation}
Indeed, assume that $\Lambda$ is achieved at an another point $y_{0}\in M$, i.e. $\Lambda=\lambda_{1}(y_{0})$. From the maximality of $Q(x_{0})$, that is, $Q(y_{0})\leq  Q(x_{0})$, we infer from \eqref{KD}, \eqref{xieta},  and \eqref{varphi} that
\begin{equation*}
\begin{split}
2&\leq  2\sqrt{\frac{\lambda_{1}(x_{0})}{\Lambda}}+\mathrm{osc}_{M}\,\varphi(|\nabla u|_{\omega}^{2}) +\mathrm{osc}_{M}\,\psi(u) \leq  2\sqrt{\frac{\lambda_{1}(x_{0})}{\Lambda}}+\frac{2}{\beta}\log 2\,.
\end{split}
\end{equation*}
This gives us
\begin{equation}\label{bei shu}
\sqrt{\lambda_{1}}(x_{0})\leq  \sqrt{\Lambda}\leq  \beta\sqrt{\lambda_{1}}(x_{0})\,.
\end{equation}
and thus yields the desired estimate \eqref{lambda and nabla}.   

Around the point $x_{0}$, we choose a a local normal coordinate system $\{e_{i}\}_{i=1}^{n}$ with respect to $\omega$  such that, at that point,
\begin{equation*} 
 g_{i\overline{j}}=\delta_{ij}\,, \qquad g_{i\overline{j}}+u_{i\overline j}=\delta_{ij}\lambda_{i}\,.
\end{equation*}
It then follows from \eqref{eigenvalue decreasing} that at $x_{0}$
\begin{equation}\label{F ii}
F^{1\overline{1}} \leq  F^{2\overline{2}} \leq  \cdots \leq  F^{n\overline{n}}\,.
\end{equation}
Utilizing the perturbation argument outlined in \cite{Sze}, we may further assume $\lambda_{1}(x_0)> \lambda_{2}(x_0)$ and thus $Q$ is smooth near $x_0$. At $x_{0}$, the maximum principle shows that
\begin{equation}\label{5..10}
\frac{(\lambda_{1})_{i}}{\sqrt{\Lambda\lambda_{1}}}=-\varphi'(|\nabla u|_{\omega}^{2})_{i}-\psi' u_{i}\,, \quad
\text{for $i=1, \cdots,n$},
\end{equation}
and
\begin{equation}\label{L Q}
\begin{split}
0 \geq  L(Q)  
= \frac{L(\lambda_{1})}{\sqrt{\Lambda\lambda_{1}}} &- F^{i\overline{i}}\frac{|(\lambda_{1})_{i}|^{2}}{2\sqrt{\Lambda\lambda_{1}}\lambda_{1}}
+\varphi'L(|\nabla u|_{\omega}^{2})
 \\[1mm]
& +\varphi'' F^{i\overline{i}}|(|\nabla u|_{\omega}^{2})_{i}|^{2}+\psi'L(u)+\psi'' F^{i\overline{i}}|u_{i}|^{2}.
\end{split}
\end{equation}

Starting from now, all calculations will be carried at the point $x_{0}$. We will adhere to the Einstein summation convention. The symbol $C_{0}$ will be used to represent a uniform constant that solely depends on $n$,  $k$, $\sup_{M}h$, $\|\nabla\sqrt[k]{h}\|_{C^{1}}$, $\mathrm{osc}_{M}u$, and $\omega$.

\subsection{Proof of Theorem \ref{main thm} when  $\mathcal{R}_{\omega}>0$}\label{Continued proof of Theorem}

We begin by establishing the following key inequality.
\begin{proposition}\label{lower bound of L Q}
At $x_{0}$, we have
\begin{equation}\label{prop}
\begin{split}
C_{0}\Big(\frac{1}{K}+\frac{\mathcal{F}}{\lambda_{1}}\Big)+\frac{\beta+1}{\beta}\mathcal{R}_{\omega}\mathcal{F}
& \geq 
G+B + \varphi'F^{i\overline{i}}\Big(|u_{i\overline i}|^{2}+\sum_{p=1}^{n}|u_{ip}|^{2}\Big)  \\[1mm]
& \qquad
+\varphi'' F^{i\overline{i}}|(|\nabla u|_{\omega}^{2})_{i}|^{2}+
\psi'L(u)+\psi'' F^{i\overline{i}}|u_{i}|^{2}\,,
\end{split}
\end{equation}
where the terms $G$, $B$ and $\mathcal{F}$ are given by
\[
\begin{split}
G &= -\frac{1}{\sqrt{\Lambda\lambda_{1}}} F^{i\overline{j},k\overline{l}}u_{i\overline{j}1} u_{k\overline{l}\overline{1}}\,,\\[2mm]
B &= -F^{i\overline{i}}\frac{|(\lambda_{1})_{i}|^{2}}{2\sqrt{\Lambda\lambda_{1}}\lambda_{1}}\,,\\[2mm]
\mathcal{F} &= \sum_{i=1}^{n}F^{i\overline{i}}\,.
\end{split}
\]
\end{proposition}

Notice that, owing to the concavity of $\sigma_{k}^{\frac{1}{k}}$ on $\Gamma_{k}$, the term $G$ is nonnegative. In order to prove Proposition \ref{lower bound of L Q},  it is essential to estimate the lower bounds for $L(\lambda_{1})$ and $L(|\nabla u|_{\omega}^{2})$ separately.

\begin{lemma}\label{lower bound of L lambda1}
At $x_{0}$, we have
\[
L(\lambda_{1}) \geq 
-F^{i\overline{j},k\overline{l}}u_{i\overline{j}1} u_{k\overline{l}\overline{1}}
-\mathcal{R}_{\omega}\lambda_{1}\mathcal{F}-C_{0}\mathcal{F}.
\]
\end{lemma}
\begin{proof}
For completeness, the comprehensive proof outlined in \cite{Sze} is presented here. The renowned formulas for computing the derivatives of the eigenvalue  $\lambda_1$  at a diagonal matrix possessing distinct eigenvalues are given below (refer to \cite{Gerhardt,Spruck} for further details):
\begin{equation}\label{derivatives of lambda}
\begin{split}
\lambda_{1}^{pq}& =\delta_{1p}\delta_{1q}\,,\\
\lambda_{1}^{pq,rs}& =(1-\delta_{1p})\frac{\delta_{1q}\delta_{1r}\delta_{ps}}{\lambda_{1}-\lambda_{p}}+(1-\delta_{1r})\frac{\delta_{1s}\delta_{1p}\delta_{rq}}{\lambda_{1}-\lambda_{r}}\,,
\end{split}
\end{equation}
where $\lambda_{1}^{pq}$ represents the derivative with respect to the entry located at the $p$-th row and $q$-th column. Let us define a matrix
\[
A_{pq}=\delta_{pq}+g^{p\overline{i}}u_{q\overline{i}}\,.
\]
It follows from \eqref{derivatives of lambda} along with the K\"ahlerity of $\omega$ that
\begin{equation}
\begin{split}
\lambda_{1,i}&=\lambda_{1}^{pq}(A_{pq})_{i}=u_{1\overline{1}i}\,,\\
\lambda_{1,i\overline{i}}&=\lambda_{1}^{pq,rs}(A_{pq})_{i}(A_{rs})_{\overline i}+\lambda_{1}^{pq}(A_{pq})_{i\overline{i}}\\
&=g_{1\overline{1}i\overline i}+u_{1\overline{1}i\overline i}+\sum_{p>1}\frac{|u_{p\overline 1i}|^{2}+|u_{1\overline{p}i}|^{2}}{\lambda_{1}-\lambda_{p}}\,.
\end{split}
\end{equation}
By commuting the covariant derivatives, we have
\begin{equation}\label{nonnegative 1}
\begin{split}
u_{1\overline{1}i\overline i}&=u_{i\overline i1\overline{1}}+(u_{1\overline{1}}-u_{i\overline i})R_{1\overline 1i\overline{i}}\\
&=u_{i\overline i1\overline{1}}+(\lambda_{1}-\lambda_{i})R_{1\overline 1i\overline{i}}\\
&\geq  u_{i\overline i1\overline{1}}+(\lambda_{1}-\lambda_{i})\inf_{l\neq 1}R_{1\overline 1l\overline{l}}\\
&\geq  u_{i\overline i1\overline{1}}-\mathcal{R}_{\omega}(\lambda_{1}-\lambda_{i})\,.
\end{split}
\end{equation}
It follows that
\begin{equation}\label{AA}
\begin{split}
F^{i\overline i}\lambda_{1,i\overline{i}}\geq  &F^{i\overline i}u_{i\overline i1\overline{1}}-\mathcal{R}_{\omega}\sum_{i=1}^{n}F^{i\overline i}(\lambda_{1}-\lambda_{i})+F^{i\overline i}g_{1\overline{1}i\overline i}\,.
\end{split}
\end{equation}
Applying $e_{1}e_{\overline{1}}$  to the equation \eqref{complex Hessian equation 1}, we see that
\begin{equation*}
F^{i\overline i}(g_{i\overline{i}1\overline{1}}+u_{i\overline{i}1\overline{1}})+F^{i\overline{j},k\overline{l}}u_{i\overline{j}1} u_{k\overline{l}\overline{1}}=(\sqrt[k]{h})_{1\overline{1}}\,.
\end{equation*}
Inserting it into \eqref{AA}, we infer that
\begin{equation*} 
\begin{split}
F^{i\overline i}\lambda_{1,i\overline{i}}\geq  &-F^{i\overline{j},k\overline{l}}u_{i\overline{j}1} u_{k\overline{l}\overline{1}}-\mathcal{R}_{\omega}\sum_{i=1}^{n}F^{i\overline i}(\lambda_{1}-\lambda_{i})+F^{i\overline i}(g_{1\overline{1}i\overline i}-g_{i\overline{i}1\overline{1}})+(\sqrt[k]{h})_{1\overline{1}}\\
\geq  & -F^{i\overline{j},k\overline{l}}u_{i\overline{j}1} u_{k\overline{l}\overline{1}}-\mathcal{R}_{\omega}\lambda_{1}\mathcal{F}-C_{0}\mathcal{F}\,.
\end{split}
\end{equation*}
Here, we have used the facts that $\mathcal{F}\geq  1$ and the homogeneity of $F$ in the last inequality. Indeed,  at the extremal point of $Q$, it can be deduced from   \eqref{Fik} and the Newton-Maclaurin inequality that
\begin{equation}\label{F dayu 1}
\mathcal{F}=\frac{1}{k}\sigma_{k}(\lambda)^{\frac{1-k}{k}}\sum_{i=1}^n \frac{\partial \sigma_{k}(\lambda)}{\partial \lambda_i}=\frac{n-k+1}{k}\sigma_{k}(\lambda)^{\frac{1-k}{k}}\sigma_{k-1}(\lambda)\geq  1\,.
\end{equation}
This completes the proof of the lemma.
\end{proof}
\begin{lemma}\label{lower bound of L de u}
At $x_{0}$, we have
\begin{equation*}
L(|\nabla u|_{\omega}^{2})\geq  F^{i\overline{i}}\Big(|u_{i\overline i}|^{2}+\sum_{p=1}^{n}|u_{ip}|^{2}\Big)-C_{0}K-\mathcal{R}_{\omega}K^{2}\mathcal{F}.
\end{equation*}
\end{lemma}
\begin{proof}
A direct calculation yields that
\begin{equation*} 
L(|\nabla u|_{\omega}^{2})=   F^{i\overline{i}}\Big(|u_{i\overline i}|^{2}+\sum_{p=1}^{n}|u_{ip}|^{2}\Big)+\sum_{p=1}^{n}F^{i\overline i}(u_{pi\overline i}u_{\overline p}+u_{\overline pi\overline i}u_{p})\,.
\end{equation*}
Therefore, we are left to show
\begin{equation}\label{reduced}
\sum_{p=1}^{n}F^{i\overline i}(u_{pi\overline i}u_{\overline p}+u_{\overline pi\overline i}u_{p})\geq  -C_{0}K-\mathcal{R}_{\omega}K^{2}\mathcal{F}\,.  
\end{equation}
By commuting the covariant derivatives,
\[
u_{pi\overline i}=u_{i\overline ip}+\sum_{q=1}^{n}u_{q}R_{i\overline ip\overline{q}}\,, \qquad u_{\overline pi\overline i}=u_{i\overline i\overline p}\,.
\]
Applying $e_{p}$ and $e_{\overline p}$ to the equation \eqref{complex Hessian equation 1}, it follows that
\begin{equation*}
F^{i\overline i}u_{i\overline{i}p}=(\sqrt[k]{h})_{p}\,, \qquad F^{i\overline i}u_{i\overline{i}\overline p}=(\sqrt[k]{h})_{\overline p}\,.
\end{equation*}
Thus, we infer that
\begin{equation}\label{nonnegative 2}
\begin{split}
\sum_{p=1}^{n}F^{i\overline i}(u_{pi\overline i}u_{\overline p}+u_{\overline pi\overline i}u_{p})
&= 2\sum_{p=1}^{n}\mathrm{Re}\Big((\sqrt[k]{h})_{p}u_{\overline{p}} \Big)+F^{i\overline{i}}\sum_{p,q=1}^{n}u_{\overline{p}}u_{q}R_{i\overline ip\overline{q}}\\
&\geq  -C_{0}K-\mathcal{R}_{\omega}K^{2}\mathcal{F}\,.
\end{split}
\end{equation}
This establishes the desired inequality \eqref{reduced}, thereby concluding the proof of the lemma.
\end{proof}

We will utilize the aforementioned computations to prove Proposition \ref{lower bound of L Q}.

\begin{proof}[Proof of Proposition \ref{lower bound of L Q}]
By combining \eqref{L Q}, Lemma \ref{lower bound of L lambda1}, and Lemma \ref{lower bound of L de u}, we derive
\begin{equation*}
\begin{split}
0\geq  {} &-\frac{1}{\sqrt{\Lambda\lambda_{1}}}F^{i\overline{j},k\overline{l}}u_{i\overline{j}1} u_{k\overline{l}\overline{1}}
-F^{i\overline{i}}\frac{|(\lambda_{1})_{i}|^{2}}{2\sqrt{\Lambda\lambda_{1}}\lambda_{1}}-C_{0}\Big(\frac{1}{K}+\frac{\mathcal{F}}{\lambda_{1}}\Big)\\
& +\varphi'F^{i\overline{i}}\Big(|u_{i\overline i}|^{2}+\sum_{p=1}^{n}|u_{ip}|^{2}\Big)
+\varphi'' F^{i\overline{i}}|(|\nabla u|_{\omega}^{2})_{i}|^{2}\\
& +\psi'L(u)+\psi'' F^{i\overline{i}}|u_{i}|^{2}
-\mathcal{R}_{\omega}\Big(\sqrt{\frac{\lambda_{1}}{\Lambda}}+\varphi'K^{2}\Big)\mathcal{F}\,.
\end{split}
\end{equation*}
The first and second terms are  $G$ and $B$ respectively. 
It follows from \eqref{xieta} and \eqref{bei shu} that
\[
\varphi'K^{2}\leq  \frac{1}{\beta}\,, \qquad \sqrt{\frac{\lambda_{1}}{\Lambda}}\leq  1\,.
\]
Combining the above inequalities, we obtain Proposition \ref{lower bound of L Q}.
\end{proof}

We now prove Theorem \ref{main thm} for the case $\mathcal{R}_{\omega} > 0$ by considering two cases:\medskip

\noindent {\bf Case 1.} Assume that $\lambda_{n}<-\delta \lambda_{1}$ for some uniform constant $\delta\in (0,\frac{1}{4}] $ to be determined soon. It follows from \eqref{5..10} and the Cauchy-Schwarz inequality that
\begin{equation*}
\begin{split}
B = {} & -F^{i\overline{i}}\frac{|(\lambda_{1})_{i}|^{2}}{2\sqrt{\Lambda\lambda_{1}}\lambda_{1}} 
\geq    - \sqrt{\frac{\Lambda}{\lambda_{1}}}(\psi')^{2}F^{i\overline i}|u_{i}|^{2}- \sqrt{\frac{\Lambda}{\lambda_{1}}} (\varphi')^{2}F^{i\overline{i}}|(|\nabla u|_{\omega}^{2})_{i}|^{2}.
\end{split}
\end{equation*}
Substituting this into (\ref{prop}), and using   \eqref{xieta} and \eqref{varphi}, we obtain
\begin{equation*}
\begin{split}
& C_{0}\Big(\frac{1}{K}+\frac{\mathcal{F}}{\lambda_{1}}\Big)+\frac{\beta+1}{\beta}\mathcal{R}_{\omega}\mathcal{F} \\
& \geq 
 \varphi'F^{i\overline{i}}\Big(|u_{i\overline i}|^{2}+\sum_{p=1}^{n}|u_{ip}|^{2}\Big) +\Big(\varphi''- \sqrt{\frac{\Lambda}{\lambda_{1}}} (\varphi')^{2} \Big) F^{i\overline{i}}|(|\nabla u|_{\omega}^{2})_{i}|^{2} \\
&\quad +\psi'L(u)+\Big( \psi''- \sqrt{\frac{\Lambda}{\lambda_{1}}}(\psi')^{2} \Big) F^{i\overline{i}}|u_{i}|^{2}\\
&\geq  \frac{\lambda_{n}^{2}}{4nK^{2}}\mathcal{F}+
\psi'L(u)\,.
\end{split}
\end{equation*}
where in the last inequality we have utilized the fact that the second and fourth terms are positive. For the last term, due to the homogeneity of $F$, we have
\begin{equation}\label{upper of Lu} 
L(u) = \sum_{i=1}^{n}F^{i\overline{i}}(\lambda_{i}-1)
=\sqrt[k]{h}-\mathcal{F}\leq   \sqrt[k]{h}
\end{equation}
and thus
\begin{equation*} 
\frac{\lambda_{n}^{2}}{4nK^{2}}\mathcal{F}\leq  -\psi'\sqrt[k]{h}+C_{0}\Big(\frac{1}{K}+\frac{\mathcal{F}}{\lambda_{1}}\Big)+\frac{\beta+1}{\beta}\mathcal{R}_{\omega}\mathcal{F}\,.
\end{equation*}
Coupling with \eqref{F dayu 1}, it follows that
\begin{equation*}
\lambda_{1}\leq  \frac{1}{\delta}(-\lambda_{n})\leq  C_{0}K\,.
\end{equation*}
This ends the proof of Case 1.
\medskip

\noindent {\bf Case 2.} Assume that $\lambda_{n}\geq -\delta \lambda_{1}$.
Let us define an index set
\begin{equation*}
I=\Big\{2\leq  i\leq  n:
F^{i\overline{i}}>\delta^{-1} F^{1\overline{1}} \ \,  \text{at $x_{0}$}\Big\}\,.
\end{equation*}
Now we decompose the bad term $B$ into two terms based on $I$:
\begin{equation*}
\begin{split}
B &  = -\sum_{i\not\in I}F^{i\overline{i}}\frac{|(\lambda_{1})_{i}|^{2}}{2\sqrt{\Lambda\lambda_{1}}\lambda_{1}}
-\sum_{i\in I}F^{i\overline{i}}\frac{|(\lambda_{1})_{i}|^{2}}{2\sqrt{\Lambda\lambda_{1}}\lambda_{1}} =: B_{1}+B_{2}\,.
\end{split}
\end{equation*}

We first deal with the term $B_{1}$. Using \eqref{5..10} along with the Cauchy-Schwarz inequality, it follows that
\begin{equation*}
\begin{split}
B_{1}&=-\frac{1}{2}\sqrt{\frac{\Lambda}{\lambda_{1}}}\sum_{i\not\in I}F^{i\overline{i}}\Big|\varphi'(|\nabla u|_{\omega}^{2})_{i}+\psi'u_{i}\Big|^{2}\\
&\geq  - \sqrt{\frac{\Lambda}{\lambda_{1}}}(\varphi')^{2}\sum_{i\not\in I}F^{i\overline{i}}|(|\nabla u|_{\omega}^{2})_{i}|^{2}- \sqrt{\frac{\Lambda}{\lambda_{1}}}(\psi')^{2}\sum_{i\not\in I}F^{i\overline{i}}|u_{i}|^{2}\\
&\geq  -\varphi''\sum_{i\not\in I}F^{i\overline{i}}|(|\nabla u|_{\omega}^{2})_{i}|^{2}-\psi''\sum_{i\not\in I}F^{i\overline{i}}|u_{i}|^{2}\,,
\end{split}
\end{equation*}
where we used \eqref{xieta} in the last line.
We then have the following inequality:
\begin{equation}\label{no B1}
\begin{split}
C_{0}\Big(\frac{1}{K}+\frac{\mathcal{F}}{\lambda_{1}}\Big)+\frac{\beta+1}{\beta}\mathcal{R}_{\omega}\mathcal{F}
& \geq 
G+B_{2} + \varphi'F^{i\overline{i}}\Big(|u_{i\overline i}|^{2}+\sum_{p=1}^{n}|u_{ip}|^{2}\Big)  \\[1mm]
& \qquad
+\varphi'' \sum_{i\in I} F^{i\overline{i}}|(|\nabla u|_{\omega}^{2})_{i}|^{2}+
\psi'L(u)+\psi'' \sum_{i\in I}F^{i\overline{i}}|u_{i}|^{2}\,.
\end{split}
\end{equation}
For the fourth term, utilizing the elementary inequality
\[
(a+b)^{2}\geq  \delta a^{2}-\frac{\delta}{1-\delta}b^{2} 
\]
where $\delta$ is sufficiently small. Furthermore, in conjunction with \eqref{bei shu}, we obtain
\begin{equation*}
\begin{split}
\varphi'' \sum_{i\in I} F^{i\overline{i}}|(|\nabla u|_{\omega}^{2})_{i}|^{2}&= \beta \sum_{i\in I} F^{i\overline{i}}\Big|\frac{(\lambda_{1})_{i}}{\sqrt{\Lambda\lambda_{1}}}+\psi'u_{i}\Big|^{2}\\
&\geq  \beta \delta \sum_{i\in I}F^{i\overline{i}}\frac{|(\lambda_{1})_{i}|^{2}}{ \Lambda\lambda_{1}} -\frac{\beta \delta}{1-\delta}(\psi')^{2} \sum_{i\in I}F^{i\overline{i}}|u_{i}|^{2}\\
&\geq  2\delta \sum_{i\in I}F^{i\overline{i}}\frac{|(\lambda_{1})_{i}|^{2}}{2\sqrt{\Lambda\lambda_{1}}\lambda_{1}} -\frac{\beta \delta}{1-\delta}(\psi')^{2} \sum_{i\in I}F^{i\overline{i}}|u_{i}|^{2}.
\end{split}
\end{equation*}
Inserting it into \eqref{no B1}, we arrive at
\begin{equation*}
\begin{split}
&C_{0}\Big(\frac{1}{K}+\frac{\mathcal{F}}{\lambda_{1}}\Big)+\frac{\beta+1}{\beta}\mathcal{R}_{\omega}\mathcal{F}\\
& \geq 
-\frac{1}{\sqrt{\Lambda\lambda_{1}}} F^{i\overline{j},k\overline{l}}u_{i\overline{j}1} u_{k\overline{l}\overline{1}} + \varphi'F^{i\overline{i}}\Big(|u_{i\overline i}|^{2}+\sum_{p=1}^{n}|u_{ip}|^{2}\Big)+
\psi'L(u)  \\
& \quad
-(1-2\delta)\sum_{i\in I}F^{i\overline{i}}\frac{|(\lambda_{1})_{i}|^{2}}{2\sqrt{\Lambda\lambda_{1}}\lambda_{1}}+\Big(\psi''-\frac{\beta \delta}{1-\delta}(\psi')^{2}\Big)\sum_{i\in I}F^{i\overline{i}}|u_{i}|^{2}\,.
\end{split}
\end{equation*}
Shrink  $\delta\leq \frac{1}{4}$, if necessary,  to ensure that 
\[
\psi''-\frac{\beta \delta}{1-\delta}(\psi')^{2}= \beta\Big(1- \frac{ \delta}{1-  \delta}\Big)(\psi')^{2}\geq  0\,.
\]
Due to the inequality (121) from \cite{Sze} or inequality (2.40) in \cite{Hou-Ma-Wu}, we can also conclude that
\begin{equation}\label{hmws}
-\frac{1}{\sqrt{\Lambda\lambda_{1}}} F^{i\overline{j},k\overline{l}}u_{i\overline{j}1} u_{k\overline{l}\overline{1}}\geq  (1-2\delta)\sum_{i\in I}F^{i\overline{i}}\frac{|(\lambda_{1})_{i}|^{2}}{\sqrt{\Lambda\lambda_{1}}\lambda_{1}}\,.
\end{equation}
For the sake of completeness, we include a proof here. Let us express $F(\omega_{u})=f(\lambda_1,\cdots,\lambda_n)$ as a symmetric function of the eigenvalues of $\omega_{u}$ with respect to $\omega$. Then at the point $x_{0}$, the first and second derivatives of $F$  are given by (see e.g. \cite{Andrews94,Gerhardt,Spruck})
\begin{equation}\label{second derive of F}
\begin{split}
F^{i\overline{j}} &= \delta_{ij}f_{i}\,;\\
F^{i\overline{k},j\overline{l}}&=f_{ij}\delta_{ik}\delta_{jl}
+\frac{f_{i}-f_{j}}{\lambda_{i}-\lambda_{j}}
(1-\delta_{ij})\delta_{il}\delta_{jk},
\end{split}
\end{equation}
where the quotient is interpreted as a limit when $\lambda_{i}=\lambda_{j}$. By the concavity of $F$, we have
\[
\frac{F^{1\overline{1}}-F^{i\overline{i}}}{\lambda_{1}-\lambda_{i}}\leq 0\,, \quad  \textrm{ for  all } i=2,\cdots,n\,.
\]
It follows from \eqref{second derive of F} and definition of $I$ that
\[
-F^{i\overline{j},k\overline{l}}u_{i\overline{j}1} u_{k\overline{l}\overline{1}}\geq -\sum_{i\in I}\frac{F^{1\overline{1}}-F^{i\overline{i}}}{\lambda_{1}-\lambda_{i}}|(\lambda_1)_{i}|^2\geq (1-\delta)\sum_{i\in I}\frac{F^{i\overline{i}}}{\lambda_{1}-\lambda_{i}}|(\lambda_1)_{i}|^2\,.
\]
Thus, it suffices to show that
\[
\frac{1-2\delta}{\lambda_1}\geq \frac{1-\delta}{\lambda_1-\lambda_i}\,, \quad \textrm{ for  all } i\in I\,,
\]
or equivalently, 
\begin{equation}\label{hmws 2}
\lambda_1\geq (2-\delta^{-1})\lambda_i\,,  \quad \textrm{ for  all } i\in I\,.
\end{equation}
Fix an arbitrary $i\in I$, we now verify \eqref{hmws 2} in two cases:
\begin{itemize}
\item[(i)] If $\lambda_i\geq 0$, then 
\[
\lambda_1-(2-\delta^{-1})\lambda_i=\lambda_1-\lambda_{i}+(\delta^{-1}-1)\lambda_i\geq 0\,.
\]
\item[(ii)] If $\lambda_i\leq 0$, then
\[
\lambda_1\geq -\delta^{-1}\lambda_n\geq -\delta^{-1}\lambda_i\geq (2-\delta^{-1})\lambda_i\,.
\]
Here, we utilize the fact that $\lambda_n$ is the smallest eigenvalue of $\omega_u$ with respect to $\omega$.
\end{itemize}
This completes the proof of \eqref{hmws 2} and hence proves the desired inequality \eqref{hmws}.

Therefore, we infer from  \eqref{upper of Lu}  and \eqref{hmws} that
\begin{equation*}
\begin{split}
&C_{0}\Big(\frac{1}{K}+\frac{\mathcal{F}}{\lambda_{1}}\Big)-\psi'\sqrt[k]{h} \geq  \Big(-\psi'-\frac{\beta+1}{\beta}\mathcal{R}_{\omega}\Big)\mathcal{F}+ \varphi'F^{i\overline{i}}\Big(|u_{i\overline i}|^{2}+\sum_{p=1}^{n}|u_{ip}|^{2}\Big)\,.
\end{split}
\end{equation*}
It can be deduced from \eqref{varphi} that 
\begin{equation}\label{split}
\begin{split}
&C_{0}\Big(\frac{1}{K}+\frac{\mathcal{F}}{\lambda_{1}}\Big)+\frac{\sqrt[k]{h}}{\beta\mathrm{osc}_{M}u} \\
&\geq  \frac{1}{2\beta}\Big(\frac{1}{ \mathrm{osc}_{M} u} -2(\beta+1)\mathcal{R}_{\omega}\Big)\mathcal{F}+ \varphi'F^{i\overline{i}}\Big(|u_{i\overline i}|^{2}+\sum_{p=1}^{n}|u_{ip}|^{2}\Big)\,.
\end{split}
\end{equation}
Under the assumption that $\mathcal{R}_{\omega}\mathrm{osc}_{M}u<\frac{1}{2(\beta+1) }$, there exists a small uniform  constant $\hat{\beta}>0$ such that
\[
\mathcal{R}_{\omega}\mathrm{osc}_{M}u\leq  \frac{1}{2(\beta+1) }-\frac{3\beta\hat{\beta}}{\beta+1}\,.
\]
It then follows that
\begin{equation}\label{eps}
\frac{1}{2\beta}\Big(\frac{1}{ \mathrm{osc}_{M} u} -2(\beta+1)\mathcal{R}_{\omega}\Big)\geq \frac{3\hat{\beta}}{\mathrm{osc}_{M}u}\,.
\end{equation}
Now we can increase $D\geq   \hat{\beta}^{-1}C_{0}\,\mathrm{osc}_{M}u$  to a sufficiently large constant to guarantee that
\begin{equation}\label{eps 1}
\frac{\hat{\beta}}{\mathrm{osc}_{M}u}\mathcal{F}\geq  \frac{\hat{\beta}}{\mathrm{osc}_{M}u}\geq  \frac{C_0}{K},
\end{equation}
where we have used \eqref{F dayu 1}. Meanwhile, without loss of generality, we may assume that
\[
\lambda_{1}\geq  \hat{\beta}^{-1}C_{0}\,\mathrm{osc}_{M}u.
\]
If this is not the case, our work is already complete. Under this assumption, we derive the following inequality
\begin{equation}\label{eps 2}
\frac{\hat{\beta}}{\mathrm{osc}_{M}u}\mathcal{F}\geq  \frac{C_0}{\lambda_{1}}\mathcal{F}.
\end{equation}
By substituting \eqref{eps}, \eqref{eps 1}, and \eqref{eps 2} into \eqref{split},  we obtain the following inequality
\begin{equation}\label{hmw type}
\begin{split}
\frac{\sqrt[k]{h}}{\beta\mathrm{osc}_{M}u}
& \geq 
\varphi'F^{i\overline{i}} \Big(|u_{i\overline i}|^{2}+\sum_{p=1}^{n}|u_{ip}|^{2}\Big)+
\frac{\hat{\beta}}{\mathrm{osc}_{M}u}\mathcal{F}\\
& \geq 
\varphi'F^{1\overline{1}} \sum_{i=1}^{n}\Big(|u_{i\overline i}|^{2}+\sum_{p=1}^{n}|u_{ip}|^{2}\Big)+
\frac{\hat{\beta}}{\mathrm{osc}_{M}u}\mathcal{F}\,.\\
\end{split}
\end{equation}
Observe that the preceding inequality informs us that
\[
F^{i\overline{i}}\leq  \mathcal{F}\leq  \frac{\sqrt[k]{h}}{\beta\hat{\beta}}, \qquad \forall \, i=1,\cdots, n.
\]
It then follows from Lemma 2.2 in \cite{Hou-Ma-Wu} that
\[
F^{1\overline 1}\geq  C_{0}\Big(\frac{\beta\hat{\beta}}{\sqrt[k]{h}}\Big)^{k-1}.
\]
Consequently, we infer from inequality \eqref{hmw type} that
\begin{equation*}
\frac{1}{K^2} \sum_{i=1}^{n}\Big(|u_{i\overline i}|^{2}+\sum_{p=1}^{n}|u_{ip}|^{2}\Big)\leq  C_{0} \frac{\sup_{M}h}{\beta^{k}\hat{\beta}^{k-1}\mathrm{osc}_{M}u}\,.
\end{equation*}
This proves the desired inequality \eqref{goal}. \qed

\begin{remark}
According to the comprehensive proof of Theorem \ref{main thm}, it is possible to extend the analogous result to a broader class of fully nonlinear elliptic equations on compact K\"ahler manifolds. We hope to  tackle this problem  in a separate discussion at a later time.
\end{remark}

\subsection{Proof of Theorem \ref{main thm} when  $\mathcal{R}_{\omega}=0$}
In this case, the holomorphic bisectional curvature of $(M,\omega)$ is nonnegative.  As previously noted, Hou \cite{Hou} derived a priori $C^1$ estimates for complex Hessian equation \eqref{complex Hessian equation 1} by assuming nonnegative holomorphic bisectional curvature and adapting an approach of Blocki \cite{Blo07}, which involves applying the maximum principle to an auxiliary function of the form 
\[Q=\log |\nabla u|_{\omega}^2 + \psi(u).\] 
In contrast, the present work employs an argument based on interpolation theory. Nevertheless, under the same curvature assumption as in \cite{Hou}, our method also yields the $C^1$ estimates, thereby generalizing and encompassing Hou’s result. We now provide some details of the argument.

First, recall an inequality from the proof of Lemma \ref{lower bound of L lambda1}. Under the assumption of nonnegative holomorphic bisectional curvature, inequality \eqref{nonnegative 1} can be replaced by
\[
\begin{split}
u_{1\overline{1}i\overline i}
&= u_{i\overline i1\overline{1}} + (u_{1\overline{1}} - u_{i\overline i})R_{1\overline 1i\overline{i}} \\
&= u_{i\overline i1\overline{1}} + (\lambda_{1} - \lambda_{i})R_{1\overline 1i\overline{i}} \geq u_{i\overline i1\overline{1}}.
\end{split}
\]
Consequently, at $x_0$ we obtain
\begin{equation*}\label{R1}
L(\lambda_{1}) \geq -F^{i\overline{j},k\overline{l}}u_{i\overline{j}1} u_{k\overline{l}\overline{1}} - C_{0}\mathcal{F}.
\end{equation*}

Second, in the proof of Lemma \ref{lower bound of L de u}, the curvature condition also allows us to replace \eqref{nonnegative 2} with
\begin{equation*}
\begin{split}
\sum_{p=1}^{n}F^{i\overline i}\Big(u_{pi\overline i}u_{\overline p} + u_{\overline pi\overline i}u_{p}\Big)
&= 2\sum_{p=1}^{n}\mathrm{Re}\Big((\sqrt[k]{h})_{p}u_{\overline{p}} \Big) + F^{i\overline{i}}\sum_{p,q=1}^{n}u_{\overline{p}}u_{q}R_{i\overline ip\overline{q}} \geq -C_{0}K.
\end{split}
\end{equation*}
It follows that at $x_{0}$,
\begin{equation*}\label{R2}
L(|\nabla u|_{\omega}^{2}) \geq F^{i\overline{i}}\Big(|u_{i\overline i}|^{2} + \sum_{p=1}^{n}|u_{ip}|^{2}\Big) - C_{0}K.
\end{equation*}

Finally, these inequalities already imply the inequality \eqref{prop} can be replaced by
\begin{equation*}\label{prop new}
\begin{split}
C_{0}\Big(\frac{1}{K} + \frac{\mathcal{F}}{\lambda_{1}}\Big)
\geq G &+ B + \varphi'F^{i\overline{i}}\Big(|u_{i\overline i}|^{2} + \sum_{p=1}^{n}|u_{ip}|^{2}\Big) \\
& + \varphi'' F^{i\overline{i}}|(|\nabla u|_{\omega}^{2})_{i}|^{2} + \psi'L(u) + \psi'' F^{i\overline{i}}|u_{i}|^{2}\,,
\end{split}
\end{equation*}
which shows that we may take $\mathcal{R}_{\omega} = 0$ in Proposition \ref{lower bound of L Q}. The remainder proof of the desired inequality \eqref{lambda and nabla} then proceeds almost identically to the argument in Subsection \ref{Continued proof of Theorem} above by replacing $\mathcal{R}_{\omega}$ with zero, and we omit the details here.

\section{Proof of Theorem \ref{h closes 1}}\label{h closes 1 section}
For the proof, we begin with a stability theorem due to Guo, Phong, and Tong \cite{GPT23}.
\begin{theorem}\cite[Theorem 1]{GPT23}
Let $u$ and $\hat{u}$ be solutions of the following complex Hessian equations
\begin{equation*} 
\sigma_{k}(\omega_{u})=h\,, \qquad \sigma_{k}(\omega_{\hat{u}})=\hat{h}
\end{equation*}
 on a compact K\"ahler manifold $M$ with $\sup_{M}(u-\hat{u})=\sup_{M}(\hat{u}-u)$, where $h$ and $\hat{h}$ are two given non-negative functions on $M$ with 
\begin{equation*}
\int_{M}h\omega^{n}=\int_{M}\hat{h}\omega^{n}=\int_{M}\omega^{n}\,, \qquad \|h\|_{L^{q}}+\|\hat{h}\|_{L^{q}}\leq  Q 
\end{equation*}
for some constants $q>\frac{n}{k}$ and $Q$. Then, there exists a constant $\hat{C}$ depends on $n$, $k$, $\omega$, $Q$, and $q$ such that
\begin{equation}\label{stability}
\|u-\hat{u}\|_{L^{\infty}}\leq  \hat{C}\|h-\hat{h}\|_{L^{1}}^{\hat{\delta} }\,, \qquad \hat{\delta} :=\frac{1}{(n+1)\frac{q-1}{q-\frac{n}{k}}+3}\,.
\end{equation}
\end{theorem}
With this theorem at our disposal, we are now in a position to complete the proof of Theorem \ref{h closes 1}.
\begin{proof}[Proof of Theorem \ref{h closes 1}]
Without loss of generality, we may assume $\mathcal{R}_{\omega}> 0$, as in the case where $\mathcal{R}_{\omega}=0$  is trivially handled. Let $\hat{u}$ be a constant function defined as $\frac{1}{2}(\sup_{M}u+\inf_{M}u)$ and set $\hat{h}=1$ in the equation  $\sigma_{k}(\omega_{\hat{u}})=\hat{h}$.

It is straightforward to verify the following relationship
\[\sup_{M}(u-\hat{u})=\sup_{M}(\hat{u}-u)\,. \]
Therefore, we infer that
\begin{equation*}
\|u-\hat{u}\|_{L^{\infty}}=\sup_{M}(u-\hat{u})=\frac{1}{2}\mathrm{osc}_{M}u\,.
\end{equation*}
It can be deduced from \eqref{stability} that
\begin{equation}\label{stability 2}
\mathrm{osc}_{M}u\leq  2\hat{C}\|h-1\|_{L^{1}}^{\hat{\delta} }\,.
\end{equation}
Here, $\hat{C}$ is a positive constant depending on $n$, $k$, $\omega$, $q$, and $\|h\|_{L^q}$. To derive \eqref{curvature} from \eqref{stability 2}, it suffices to ensure that 
\begin{equation*}
2\hat{C}\mathcal{R}_{\omega} \|h-1\|_{L^{1}}^{\hat{\delta} }< \frac{1}{2(2+\log 2)}.
\end{equation*}
In simpler terms, this translates to 
\[
 \|h-1\|_{L^{1}}<\Big(\frac{1}{4(2+\log 2)\mathcal{R}_{\omega}\hat{C}} \Big)^{(n+1)\frac{q-1}{q-\frac{n}{k}}+3} 
\]
where the right-hand side represents our desired threshold $\epsilon_{0}$. This completes our proof of Theorem \ref{h closes 1}.
\end{proof}

\section{Proof of Theorem \ref{gradient estimate}}\label{gradient section}
By substituting $u$  for $u-\inf_{M}u$, we may assume that $\inf_{M}u=0$, thereby leading to 
\[\mathrm{osc}_{M}u=\sup_{M}u=\|u\|_{L^{\infty}}.\] 
The estimate obtained in the Section 3 is well-suited for a standard interpolation argument which allows us to obtain directly a gradient estimate and consequently, also a Laplacian estimate.  The following proposition will provide a complete proof of Theorem \ref{gradient estimate}.

\begin{proposition}\label{gradient}
Suppose that there is a uniform constant $C_{0}$ such that
\[
\|u\|_{L^\infty}\leq  C_{0}\,, \qquad \|\Delta_{\omega}  u \|_{L^\infty}\leq  C_{0}\Big( \|\nabla  u \|_{L^\infty}+1 \Big)\,.
\]
Then, there exists a uniform bound $C'$ with
\begin{equation}\label{gradient estimate }
\|\nabla u\|_{L^\infty} \leq  C'\,.
\end{equation}
\end{proposition}
\begin{proof}
The proof in the hypercomplex setting has already been established by Gentili and the first-named author in \cite{GZb}.  Interpolation theory, as outlined in section 6.8 of  \cite{GT}, demonstrates that for any constant $ \epsilon>0 $ and $\alpha$ within the range of $ 0<\alpha<1 $, there is a positive constant  $ C(\epsilon,\alpha)$  such that
\[
\|\nabla u\|_{L^\infty}\leq  C(\epsilon,\alpha) \|u\|_{L^\infty}+\epsilon\|u\|_{C^{1,\alpha}}\leq  C(\epsilon,\alpha) C_{0}+\epsilon \|u\|_{C^{1,\alpha}}\,.
\]
By selecting $ q=\frac{2n}{1-\alpha}>2n $ and  applying  to Morrey's inequality along with  elliptic $ L^q $-estimates for the Laplacian, we obtain the following inequality:
\[
\|u\|_{C^{1,\alpha}}\leq  C_{1} \|u\|_{W^{2,q}}\leq  C_{2}\Big( \|u\|_{L^q}+\|\Delta_{\omega}  u\|_{L^q} \Big)\leq  C_{2}\Big(\|u\|_{L^\infty}+\|\Delta_{\omega}  u\|_{L^\infty}\Big)\,.
\]
Here, $ C_{1} $ and $C_{2}$ are positive constants that depend only on $ \alpha $.

Combining all the preceding inequalities, we conclude that
\[
\|\nabla u\|_{L^\infty}\leq  C(\epsilon,\alpha) C_0+\epsilon C_{0}C_{2}\big( 2+\|\nabla u\|_{L^\infty} \big)\,. 
\]
From this inequality, we can deduce the gradient estimate \eqref{gradient estimate } by selecting $\epsilon$ small enough to ensure that $\epsilon<(C_{0}C_{2})^{-1} $. Now, we can select 
\[ 
C'=\frac{ C_{0}}{1-\epsilon C_{2}C_{0}}\big(C(\epsilon,\alpha)+2\epsilon C_{2}\big)
\]
as a candidate for the desired upper bound of $\|\nabla u\|_{L^\infty}$. This completes the proof of proposition.
\end{proof}

\section*{Acknowledgements}
The first author was supported by the National Natural Science Foundation of China (Grant No. 12401065) and the Initial Scientific Research Fund of Young Teachers in Hefei University of Technology (Grant No. JZ2024HGQA0119). The second author was supported by the National Key R and D Program of China 2020YFA0713100 and by NSF in China, Nos. 12141104, 11801535, and 11721101. \medskip

\noindent\textbf{Data availability} {This manuscript has no associated data.}
\section*{Declarations}
\noindent\textbf{Conflict of interest} The authors state that there is no conflict of interest.

\end{document}